\documentclass[11pt]{article}
\usepackage[utf8]{inputenc}
\usepackage[a4paper]{geometry}
\usepackage[english]{babel}
\usepackage[super]{nth}
\usepackage{amsmath, amsfonts, amsthm, physics, enumitem}
\usepackage{titlesec}
\usepackage{tikz} 
\usetikzlibrary{positioning}
\usepackage{mathtools}
\DeclarePairedDelimiter\floor{\lfloor}{\rfloor}
\usetikzlibrary{patterns,arrows,decorations.markings}
\usepackage{authblk}
\usepackage{tikz-cd}
\usepackage{centernot}

\makeatletter
\newtheorem*{rep@theorem}{\rep@title}
\newcommand{\newreptheorem}[2]{%
\newenvironment{rep#1}[1]{%
 \def\rep@title{#2 \ref{##1}}%
 \begin{rep@theorem}}%
 {\end{rep@theorem}}}
\makeatother

\newtheorem{thm}{Theorem}[section]
\newtheorem{prop}[thm]{Proposition}
\newtheorem{cor}[thm]{Corollary}

\newtheorem{q}[thm]{Question}
\newtheorem{rmk}[thm]{Remark}

\newreptheorem{thm}{Theorem}

\theoremstyle{definition}

\newcommand{\midarrow}{\tikz \draw[-triangle 90] (0,0) -- +(.1,0);}

\newcommand{\KT}{\mathrm{KT}^{4}}

\newcommand{\der}[1]{\frac{\partial}{\partial#1}}
\newcommand{\dif}[1]{\mathrm{d}#1}
\newcommand{\erfc}{\mathrm{erfc}}

\title{Harmonic Forms on the Kodaira-Thurston Manifold}
\author{Tom Holt\thanks{Thomas.Holt@warwick.ac.uk}\,   and Weiyi Zhang\thanks{Weiyi.Zhang@warwick.ac.uk}}
\affil{Mathematics Institute,  University of Warwick, Coventry, CV4 7AL, England}

\date{}
\begin{document}

\maketitle

\begin{abstract}
We introduce an effective method to determine the $\bar\partial$-harmonic forms on the Kodaira-Thurston manifold endowed with an almost complex structure and an almost Hermitian metric.  Using the Weil-Brezin transform, we reduce the elliptic PDE system to countably many linear ODE systems. By solving a fundamental problem on linear ODE systems, the problem of finding $\bar\partial$-harmonic forms is equivalent to a generalised Gauss circle problem. 

We demonstrate two remarkable applications. First, the dimension of the almost complex $\bar\partial$-Hodge numbers on the Kodaira-Thurston manifold could be arbitrarily large. Second, Hodge numbers vary with different choices of almost Hermitian metrics. This answers a question of Kodaira and Spencer in Hirzebruch's 1954 problem list.
\end{abstract}

\tableofcontents

\section{Introduction}
Hodge theory is a method introduced by Hodge in the 1930s to study the cohomology groups of compact manifolds using the theory of elliptic partial differential equations. Not only has Hodge theory since become part of the standard repertoire in algebraic geometry, particularly through its connection to the study of algebraic cycles, but also the elliptic theory has become a fundamental tool to study the topology of manifolds. 

The most fundamental idea in classical Hodge theory for complex manifolds is the introduction of the finite dimensional vector spaces of $\bar\partial$-harmonic $(p, q)$-forms $\mathcal H^{p, q}$ with respect to a K\"ahler (or Hermitian) metric. When the manifold is K\"ahler, these groups give rise to a decomposition of the cohomology groups with complex coefficients. Each space $\mathcal H^{p, q}$ can be identified with a coherent sheaf cohomology group, called the Dolbeault group, which depends only on the underlying complex manifold but not on the choice of the Hermitian metric. Their dimensions $h^{p, q}$, called the Hodge numbers, are important invariants of complex manifolds.  They do not change when the complex structures are K\"ahler and vary continuously, but they are in general not topological invariants.

As discussed in \cite{Hir}, we can still define the group $\mathcal H^{p, q}$ for closed almost complex manifolds with almost Hermitian metrics. Precisely, the almost complex structure $J$ on $M$ induces a decomposition of the complexified cotangent bundle $T^*M\otimes \mathbb C=(T^*M)^{1, 0}\oplus (T^*M)^{0,1}$, which in turn induces a decomposition of complex differential forms into $(p,q)$-forms. We define $\bar\partial$ (respectively $\partial$) to be the component of the exterior derivative $d$ that raises $q$ (respectively $p$) by one. Notice we no longer have $d=\partial+\bar\partial$ or $\bar\partial^2=0$ when $J$ is not integrable. Given an almost Hermitian metric we also define the operator $\bar{\partial}^{*}=-*\partial*$ along with the $\bar{\partial}$-Laplacian
$$\Delta_{\bar\partial}=\bar\partial\bar\partial^*+\bar\partial^*\bar\partial$$
Here $*$ denotes the Hodge star operator.
The space $\mathcal H^{p,q}$ is defined to be the kernel of $\Delta_{\bar\partial}$ in the space of $(p, q)$-forms. When the manifold is compact, as the notation would suggest, $\bar{\partial}^{*}$ is the $L^2$ adjoint of $\bar\partial$ with respect to an almost Hermitian metric and so we have
  $\ker\Delta_{\bar\partial}=\ker\bar\partial\cap\ker\bar\partial^*$. Using the initial definition of $\bar{\partial}^{*}$, this is equivalent to $\mathcal H^{p,q}=\ker\bar\partial\cap \ker\partial *$. See \cite{CZ} for a detailed treatment.

Although the almost complex Hodge theory could be very useful, in particular for geometrically interesting almost complex structures, not much is known beyond the attempts to develop harmonic theory for almost K\"ahler manifolds by Donaldson \cite{Don}, for strictly nearly K\"ahler $6$-manifolds by Verbitsky \cite{Ver}, and very recently the introduction of a variant of $\mathcal H^{p,q}$ using $\bar\partial$-$\mu$-harmonic forms by Cirici and Wilson \cite{CW}. Moreover, few non-trivial examples of $\mathcal H^{p,q}$ for non-integrable almost complex structures have been computed. 

In this paper, we offer an effective method to study $\bar\partial$-harmonic forms, or more generally other linear PDE systems. For example this method can be applied to any torus bundle over $S^1$, but here we will restrict our attention to the Kodaira-Thurston manifold, which was the first non-K\"ahler example to be found that admits both complex (due to Kodaira) and symplectic (due to Thurston) structures. Its definition is recalled in Section \ref{KTM}. 

There are two remarkable features of the $\mathcal H^{p,q}$ and their dimensions $h^{p,q}$ that grow out of our computation. First, as mentioned above, the Hodge numbers are constant in a small neighbourhood of the moduli of a given K\"ahler manifold \cite{Voi}, in particular they do not change when the complex structures are K\"ahler and vary continuously. Moreover, even if Hodge numbers are in general not topological invariants, we know they are bounded, for example by the Betti numbers, for a fixed compact complex manifold with K\"ahler structures. 

However, neither statement is true in the almost complex setting when we vary almost complex structures in an almost K\"ahler family. The following is derived from Corollary \ref{anyint} and Proposition \ref{sol>0}.

\begin{thm}\label{intro1}
There is a continuous family of non-integrable almost complex structures $J_{a,b}$, $a,b\in \mathbb R$, $b\neq 0$, on the Kodaira-Thurston manifold whose $h^{0,1}_{J_{a,b}}=h^{2,1}_{J_{a,b}}$ are computed using certain almost K\"ahler metrics. For any $n\in\mathbb Z^+$ such that $8\centernot\mid n$, there is a choice of $a$ and $b$ such that $h^{0,1}_{J_{a,b}}=n$.
\end{thm}

Second, we mentioned that for integrable complex structures, $\mathcal H^{p, q}$ depends only on the underlying complex manifold but not on the choice of Hermitian metric. In the almost complex setting, there is the following famous question of Kodaira and Spencer which appeared as Problem 20 in Hirzebruch's 1954 problem list \cite{Hir}.

\begin{q}[Kodaira-Spencer]\label{KS}
Let $M$ be an almost complex manifold. Choose an almost Hermitian structure and consider the numbers $h^{p,q}$. Is $h^{p,q}$ independent of the choice of the almost Hermitian structure?
\end{q}
According to a recent update of Hirzebruch's problem list \cite{Kot}, there seems to have been no progress at all on this problem, besides the work \cite{Don, Ver, CW} mentioned above. For some special value of pairs, for example when $q=0$ or on compact manifolds when $q=\dim M$ (see \cite{CZ}), Question \ref{KS} was answered affirmatively. Moreover, Question \ref{KS} is among the $5$ widely open ones in the list of $34$ problems in \cite{Hir}, alongside ones like the (non-)existence of complex structures on $S^6$ and the classification of complex structures on $\mathbb CP^n$.

Using the same family of almost complex structures $J_{a,b}$ as in Theorem \ref{intro1} but changing the almost Hermitian metric, in Section \ref{Sec KS} we are able to give a negative answer to this question.

\begin{repthm}{KSKT}
There exist almost complex structures on the Kodaira-Thurston manifold such that $h^{0,1}$ varies with different choices of almost Hermitian metrics.
\end{repthm}
The method used to obtain the above results in fact offers more precise computational results. For any specific almost complex structure and almost Hermitian metric, chosen from the above mentioned family, it is possible to find the $\bar\partial$-harmonic forms explicitly. 

The method used to find these $\bar\partial$-harmonic forms is summarised in the following diagram.
\begin{center}
\begin{tikzcd}[column sep=large]
\hbox{ODE} \arrow{ddr}[swap]{\hbox{Stokes Phenomenon}}
&   &\hbox{Number Theory}\arrow{ddl}{\hbox{Gauss circle problem}}\\
                                       &\hbox{PDE}\arrow{ur} \arrow{ul}[swap]{\hbox{ Weil-Brezin transform}}&\\
                                      &\bar\partial\hbox{-harmonic forms}\arrow{u}&
\end{tikzcd}
\end{center}

First, we have an elliptic PDE system obtained from the elliptic operator $\Delta_{\bar\partial}$ and $\ker\Delta_{\bar\partial}=\ker\bar\partial\cap\ker\partial *$. In our calculation for $h^{2,0}$ and $h^{1, 0}$ in \cite{CZ}, the equations automatically kill the dependence of solutions on a couple of variables, leaving only variables over which the solutions are periodic. This prompted Haojie Chen and the second author to solve the equations using Fourier analysis. This strategy no longer works in the calculation of $h^{0,1}$ as the coefficients of our forms will depend on all the variables in general. 

However, the Fourier theory can also be made to work for non-abelian groups although it is probably not as powerful as in the abelian case at least from a computational perspective. It could be understood as decomposing function spaces with respect to irreducible unitary representations, as in the classical theorem of Peter and Weyl for compact groups.  
As the underlying group for the Kodaira-Thurston manifold is the Heisenberg group, its irreducible unitary representations are classified by the classical Stone-von Neumann theorem. Moreover, as the Kodaira-Thurston manifold is obtained from quotienting out a discrete lattice, we eventually have a theory like the classical Fourier series, where only a discrete subset of the irreducible unitary representations comes into play. This is in fact a classical theory in Harmonic Analysis, which is related to the Weil-Brezin transform. This is adapted to our setting in Section \ref{WBsec} whose motivation is explained in Section \ref{0111}.

Applying this Fourier theory for the Heisenberg group, we are able to transform our PDE into a set of countably many first order linear ODE systems and a set of countably many zeroth order linear equations. These linear ODE systems have a very fundamental form and are probably the simplest ODE systems other than the ones with constant coefficients. However, after consulting several experts, it seems that the asymptotic behaviour at infinity of these ones are not well studied. We obtain the following general result. 

\begin{repthm}{1stode}
Let $A,B \in M_{2}(\mathbb{C})$ be matrices and let $A$ have two distinct, real eigenvalues $\lambda_{1}$, $\lambda_{2}$ with $\lambda_{1}>0>\lambda_2$ then the equation
\begin{align}
    \frac{d}{dx}
    \begin{pmatrix}
        f\\
        g
    \end{pmatrix}
    = (Ax+B)
    \begin{pmatrix}
        f\\
        g
    \end{pmatrix}\label{introe2}
\end{align}
has a pair of solutions $f,g\in L^{2}(\mathbb{R})$ if and only if the following holds: Given $P\in GL(2, \mathbb C)$ such that $PAP^{-1}$ is diagonal and writing $PBP^{-1}$ as $\begin{pmatrix}
    b_{1} & b_{2}\\
    b_{3} & b_{4}
\end{pmatrix}$ we have $b_{2}b_{3}\in (\lambda_{1}-\lambda_{2})\cdot \mathbb{Z}^-$, and in this situation both $f$ and $g$ are Schwartz functions.
\end{repthm}

 We use the classical Laplace transform, although it cannot be applied directly to \eqref{introe2} or its corresponding second order ODEs for $f$ and $g$. The key trick is a transformation changing the coefficient matrix $A$ to a simpler form.  

Our argument could be used to obtain the full set of data for Stokes phenomenon of \eqref{introe2}. In fact, our argument also applies when $A$ has two (maybe equal) complex eigenvalues $\lambda_1, \lambda_2$ with $\Re\lambda_1>0$. However, Theorem \ref{1stode} has a neat form and is sufficient for the applications in this paper. 

Theorem \ref{1stode} solves the countably many first order ODE systems. In particular, for most cases they do not have any solution. Then our search for harmonic forms is reduced to the zeroth order linear equations. Surprisingly, the solutions correspond to the integer lattice points on a circle passing through the origin in the real plane. In other words, it is essentially a generalisation of the classical Gauss circle problem. And this number theoretical counting leads to a computation of the Hodge numbers. 

\begin{repthm}{h01}
For the family of almost complex structures $J_{a, b}$, $b\neq 0$, and the standard almost K\"ahler metrics on the Kodaira-Thurston manifold, whenever $d:= \frac{b}{8\pi} \in \mathbb{Q}$, the Hodge number $h^{0,1}$ is equal to the number of integer pairs $(l,m)$ solving the generalised Gauss circle problem
$$(l-d)^{2}+m^{2}=d^{2}.$$

Furthermore, if $d =  \frac pq$, with $\gcd(p,q)=1$ and $q \le 5$, we have 
$$
h^{0,1}=\begin{cases} \begin{array}{cll}
4(\beta_{1}+1)(\beta_{2}+1)\dots (\beta_{t}+1)& \hbox{ if $q=1$},\\
2(\beta_{1}+1)(\beta_{2}+1)\dots (\beta_{t}+1)& \hbox{ if $q=2$},\\
(\beta_{1}+1)(\beta_{2}+1)\dots (\beta_{t}+1)& \hbox{ if $q=3$},\\
(\beta_{1}+1)(\beta_{2}+1)\dots (\beta_{t}+1)& \hbox{ if $q=4$},\\
(\beta_{1}+1)(\beta_{2}+1)\dots (\beta_{t}+1)& \hbox{ if $q=5$}.
\end{array}
\end{cases}
$$
where $p^2$ has the prime factorisation $p^2= 2^{\alpha_{0}}p_1^{\alpha_{1}}\dots p_s^{\alpha_{s}}q_1^{\beta_{1}}\dots q_t^{\beta_{t}}$ with $p_{i}\equiv 3 \mod 4$ for all $i$ and $q_{j}\equiv 1 \mod 4$ for all $j$.
\end{repthm}

This PDE-ODE-NT method to calculate $h^{0.1}$ should be very useful in studying Hodge theory and more generally solving PDEs on almost complex manifolds. 
First, as we do not assume any symmetry of almost complex structures, it also works for other almost complex structures on the Kodaira-Thurston manifolds or more general nilmanifolds where we have Kirillov theory on irreducible unitary representations of nilpotent Lie groups. 
Moreover, it could also be used to compute the space of bundle valued harmonic forms where theta functions will be the coefficients of such forms in terms of a smooth basis. This may explain the appearance of lattice point counting from a number theory perspective as the Jacobi theta function $\vartheta_3^2(q)$  is the generating function of the classical counting of square sums. The method may also work for solving PDEs on compact quotients of some other Lie groups by lattice subgroups, or more generally geometric manifolds {\it \`a la} Thurston. 

Finally, to complete the list of Hodge numbers for the Kodaira-Thurston manifold, we need to compute $h^{1,1}$. We show in Proposition \ref{akh11} that $h^{1,1}$ is a topological invariant for any almost complex structure with any almost K\"ahler metric. In particular, applying this to our $J_b$, we have $h^{1,1}=3$. 

\subsection*{Acknowledgement} We would like to thank Haojie Chen for suggesting that we study the Kodaira-Spencer question using equations similar to \eqref{V1} and \eqref{V2}. This  greatly motivates us to solve these equations explicitly. We are grateful to Mario Micallef and Christoph Ortner for  stimulating discussions when we are baffled by the countably many linear ODE systems. These discussions gradually made it clear that we should aim to prove a result like Theorem \ref{1stode}. Finally, we thank the referee for helpful suggestions which improve the presentation of the paper.

\section{The Kodaira-Thurston Manifold}\label{KTM}
We first recall the definition of the Kodaira-Thurston manifold and define a family of non-integrable almost complex structures on it. 

The Kodaira-Thurston manifold $\KT$ is defined to be the direct product $S^{1}\times (H_{3}(\mathbb{Z})\backslash H_{3}(\mathbb{R}))$, where $H_{3}(\mathbb{R})$ denotes the Heisenberg group
$$H_{3}(\mathbb{R}) = \left\{ \begin{pmatrix}
    1 & x & z\\
    0 & 1 & y\\
    0 & 0 & 1
\end{pmatrix}\in GL(3,\mathbb{R}) \right\},$$
and $H_{3}(\mathbb{Z})$ is the subgroup  $H_3(\mathbb R)\cap GL(3,\mathbb{Z})$ acting on $H_{3}(\mathbb{R})$ by left multiplication. We call and $H_{3}(\mathbb{Z})\backslash H_{3}(\mathbb{R})$ the Heisenberg manifold. It is also useful to consider the covering of this manifold by $\mathbb{R}^{4}$ given by identifying points with the relation
\begin{equation}\begin{pmatrix}
    t\\
    x\\
    y\\
    z
\end{pmatrix}\sim
\begin{pmatrix}
    t+t_{0}\\
    x+x_{0}\\
    y+y_{0}\\
    z+z_{0}+x_{0}y
\end{pmatrix}, \label{quot}\end{equation}
for every choice of integers $t_{0},x_{0},y_{0},z_{0}\in \mathbb{Z}$. 

It should be noted that $\der{y}$ is not a well-defined smooth vector field. Instead we can use vector fields $\der{t}, \der{x}, \der{y}+x\der{z}$ and $\der{z}$, which \textit{are} well defined, to form a basis at each point. In this paper, we will consider a family of non-integrable almost complex structures given by the matrix
$$J_{a,b} = \begin{pmatrix}
        0 & -1 & 0 &  0\\
        1 &  0 & 0 &  0\\
        0 &  0 & a &  b\\
        0 &  0 & c &  -a
        \end{pmatrix},$$
acting on our basis, with $a,b \in \mathbb{R}$, $b\neq 0$ and $c = -\frac{a^{2}+1}{b}$.
We can then define the vector fields
$$V_{1} = \frac 12 \left(\der{t}-i\der{x}\right)\quad \mathrm{\&}\quad V_{2} = \frac 12 \left(\left(\der{y}+x\der{z}\right)-\frac{a-i}{b}\der{z}\right), $$
spanning $T^{1,0}_{x}M$ at every point, along with their dual 1-forms
$$\phi_{1} = \dif{t}+i\dif{x}\quad\mathrm{\&}\quad \phi_{2} = (1-ai)\dif{y}-ib(\dif{z}-x\dif{y}). $$
These 1-forms satisfy the structure equations
$$d \phi_1 = 0, $$
$$d \phi_2 = \frac b4\left(\phi_1\wedge\phi_2 + \phi_1 \wedge \bar\phi_2 + \phi_2 \wedge\bar\phi_1 - \bar\phi_1 \wedge\bar\phi_2\right). $$
Notice that $d\phi_2$ has a component with bidegree $(0,2)$. 
This demonstrates that the exterior derivative cannot be written as the sum $d = \partial + \bar\partial$ and therefore the almost complex structure $J_{a,b}$ is non-integrable for all $a\in\mathbb{R}$, $b\in \mathbb{R}\setminus \{0\}$.

\subsection{The almost complex $\bar\partial$-Hodge numbers for the Kodaira-Thurston manifold}\label{hodge}
We will compute the spaces $\mathcal H^{p, q}$ of $\bar\partial$-harmonic forms on $\KT$ and their dimension $h^{p,q}$. 
We start our computation on the metric for which $V_1, V_2, \bar V_1$ and $\bar V_2$ form an orthonormal basis. We will call this the \textit{standard orthonormal metric} with respect to $J_{a,b}$. In other words, for any $b\ne 0$, the almost complex structure $J_{a, b}$ is almost K\"ahler and a compatible symplectic structure is  $\frac{i}{2}(\phi_1\wedge\bar\phi_1+\phi_2\wedge\bar\phi_2)=dt\wedge dx+bdz\wedge dy$.

Since the Kodaira-Thurston manifold is compact, by Serre duality \cite{CZ} we have the symmetry $h^{p,q}=h^{2-p,2-q}$. Hence, we only need to compute $h^{1,0}$, $h^{2, 0}$, $h^{0,1}$ and $h^{1,1}$. Specifically, Serre duality gives us $\mathcal{H}^{p,q} = *\overline{\mathcal{H}^{2-p,2-q}}$, so we can even determine the spaces $\mathcal{H}^{p,q}$ by studying only four.

\subsubsection{$\boldsymbol{h^{1,0}}$ and $\boldsymbol{h^{2,0}}$}
In \cite{CZ} Haojie Chen and the second author found that, when considering the almost complex structure $J_{a,b}$ with $a=0$ alongside the standard orthonormal metric, $h^{1,0}$ and $h^{2,0}$ are given by
$$h^{2,0} =
\begin{cases}
    1 & b \in 4\pi \mathbb{Z}, b\neq 0\\
    0 & b \not\in 4\pi \mathbb{Z}
\end{cases},$$
$$h^{1,0}=1.$$
In fact, the argument used to prove this is unchanged when $a$ is non-zero, allowing us to determine the value of $h^{2,0}$ and $h^{1,0}$ for all $a \in \mathbb{R}$, yielding the same results as above.

Moreover, it was shown in \cite{CZ} that on any almost complex manifold, $h^{2,0}$ and $h^{1,0}$ are both independent of the almost Hermitian metric used to define $\Delta_{\bar\partial}$.

\subsubsection{$\boldsymbol{h^{0,1}}$ and  $\boldsymbol{h^{1,1}}$}\label{0111}

Now, to calculate $h^{0,1}$ we first look at a general smooth $(0,1)$-form $s = f\bar{\phi}_{1}+g\bar{\phi}_{2}$ with $f,g\in C^{\infty}(\KT)$.
Requiring that $\bar{\partial}s = 0$ gives us our first condition
\begin{align*}
    \bar{\partial}(f\bar{\phi}_{1}+g\bar{\phi}_{2})&= \bar{\partial}(f)\wedge \bar{\phi}_{1}+\bar{\partial}(g)\wedge \bar{\phi}_{2}+g\bar \partial(\bar{\phi}_{2}) \\
    &= \left(-\bar{V}_{2}(f)+\bar{V}_{1}(g)+g\frac b4\right)\bar{\phi}_{1}\wedge\bar{\phi}_{2}=0.
\end{align*}
Similarly requiring that $\partial *s = 0$ gives us the second
\begin{align*}
    \partial*(f\bar{\phi}_{1}+g\bar{\phi}_{2}) &= \partial(f\phi_{2}\wedge\bar{\phi}_{1}\wedge\bar{\phi}_{2}-g\phi_{1}\wedge\bar{\phi}_{1}\wedge\bar{\phi}_{2})\\
    &= \left(V_{1}(f)+f\frac b4 - f \frac b4 +V_{2}(g)\right)\phi_{1}\wedge\phi_{2}\wedge\bar{\phi}_{1}\wedge\bar{\phi}_{2}\\
    &= \left(V_{1}(f)+V_{2}(g)\right)\phi_{1}\wedge\phi_{2}\wedge\bar{\phi}_{1}\wedge\bar{\phi}_{2}=0.
\end{align*}

So, we have shown that a $(0,1)$-form is in the kernel of $\Delta_{\bar{\partial}}$ exactly when $f$ and $g$ satisfy
\begin{align}
    -\bar{V}_{2}(f)+\bar{V}_{1}(g)+g\frac b4&=0,\label{V1}\\
    V_{1}(f)+V_{2}(g)&=0.\label{V2}
\end{align}

This elliptic system is much harder to solve directly. We are not able to use the classical Fourier series as in the computation of $\mathcal H_{\bar\partial}^{2,0}$ in \cite{CZ} as $f$ and $g$ will depend on all the variables in general. 
However, theoretically every locally compact group has a Fourier theory since its essence is to decompose Hilbert function spaces with respect to irreducible unitary representations which play the role of characters in the abelian group case. In our situation, the Kodaira-Thurston manifold is derived from the Heisenberg group, whose irreducible unitary representations are classified by the Stone-von Neumann theorem, which says any irreducible unitary representation that is non-trivial on the center of the Heisenberg group is unitarily equivalent to one of the classical representations $\rho_h$ on $L^2(\mathbb R)$ parametrised by non-zero real numbers. 

However, the Kodaira-Thurston manifold is obtained by quotienting out the $\mathbb Z$ lattice. In other words, our functions on the Heisenberg group (times a circle) are periodic with respect to the integral subgroup. Comparing this to the classical Fourier theory, the process is like going from the Fourier transform  (the $\mathbb R$ to $\mathbb R$ Fourier theory) where the characters are $e^{itx}$ to the Fourier series (the $S^1$ to $(S^1)^{\vee}=\mathbb Z$ Fourier theory) where the characters will be parametrised by a discrete subgroup $\mathbb Z$ of $\mathbb R$. This process in the Heisenberg group setting is a classical topic in harmonic analysis (see Chapter $1$, in particular Section $10$, of \cite{Fol} for a nice introduction). However, we are not aware of any practical use of this strategy in solving PDEs. In the next section, we will adapt this classical theory to our setting and reduce the PDE system \eqref{V1} and \eqref{V2} to countably many ODEs.

We will leave the discussion of $h^{1,1}$ to Section \ref{h11} where we will see in Proposition \ref{akh11} that for $J_{a,b}$ with any compatible almost K\"ahler metric, $h^{1,1}=3$.

\section{Fourier Transform on the Kodaira-Thurston Manifold}
In this section, we introduce a method to solve the PDE system \eqref{V1} and \eqref{V2}. We use harmonic analysis for the Heisenberg group to reduce the PDE system to a set of countably many ODE systems and a set of countably many systems of zeroth order linear equations. In Theorem \ref{1stode} we solve the ODE systems and then in Section \ref{n=0NT} we reduce the zeroth order linear equations to a generalised version of the Gauss circle problem.

\subsection{Decomposing functions using the Weil-Brezin transform}\label{WBsec}

In this subsection we will describe a decomposition of $L^2(\KT)$ derived from a similar decomposition on the Heisenberg manifold.  

\begin{prop}\label{propdecomp}
  The space of $L^2$ functions on the Kodaira-Thurston manifold decomposes in the following way.
  \begin{equation}\label{decomp}
      L^2(\KT) = \left(\widehat{\bigoplus_{\substack{k,m,n \in \mathbb{Z} \\ n\neq 0 \\ 0\leq m< \abs{n}}}}\mathcal{H}^{k,m,n}\right) \oplus  \left(\widehat{\bigoplus_{k,l,m\in \mathbb{Z}}}\mathcal{H}^{k, l,m,0}\right),
  \end{equation}
  where
  $$\mathcal H^{k, m, n}=\left\{ \sum_{\xi\in \mathbb{Z}}F(x+\xi)e^{2\pi i(kt+(m+n\xi)y+nz)} \,\middle|\,F\in L^2(\mathbb{R})\right\}\subset L^2(\KT), $$
  and
  $$\mathcal{H}^{k,l,m,0} := \left\{G e^{2\pi i (kt+lx+my)}\,\middle|\,G\in \mathbb{C}\right\} \subset L^2(\KT). $$
  Here $\hat \oplus$ denotes the direct sum followed by the closure with respect to the $L^2$ norm.
  \begin{proof}
  A decomposition of functions on the Heisenberg manifold already exists, (see \textit{e.g.} Section I.5 of \cite{Aus} and Section 1.10 of \cite{Fol}), allowing us to write any $f \in L^2\left(H_3(\mathbb{Z})\setminus H_3(\mathbb{R})\right)$ as
 $$f(x,y,z) = \sum_{\substack{m,n \in \mathbb{Z}\\ n \neq 0 \\ 0 \leq m < \abs{n}}}\left( \sum_{\xi \in \mathbb{Z}} F_{m,n}(x + \xi) e^{2\pi i ((m+n\xi)y +nz)}\right) + \sum_{l,m\in \mathbb{Z}} G_{l,m} e^{2\pi i (lx+my)}, $$
 with $F_{m,n}(x) \in L^2(\mathbb{R})$ and $G_{l,m} \in \mathbb{C}$ determined uniquely.
 
 Recall that $\KT := S^1 \times \left(H_3(\mathbb{Z})\setminus H_3(\mathbb{R})\right)$. Therefore, given any $f \in L^2(\KT)$, we can take a standard Fourier expansion with respect to $t$ (the variable parametrising $S^1$) to write
 $$f(t,x,y,z) = \sum_{k\in \mathbb{Z}} f_k(x,y,z) e^{2\pi i kt}, $$
 where $f_k \in L^2\left(H_3(\mathbb{Z})\setminus H_3(\mathbb{R})\right)$. The desired result can then be obtained by further expanding $f_k$ in the manner described above. \end{proof}
 \end{prop}
 On the Heisenberg manifold, the Weil-Brezin transform \cite{Aus, Fol} is a map acting on square integrable functions $F\in L^2(\mathbb{R})$ as follows
 $$F(x) \mapsto \sum_{\xi \in \mathbb{Z}}F(x+\xi)e^{2\pi i ((m+n\xi)y+nz)}. $$
 
 Analogously, on the Kodaira-Thurston manifold we have the transform
 \begin{align*}
      W_{k,m,n}: L^2(\mathbb{R}) & \rightarrow L^2(\KT)\\
      F(x) &\mapsto  \sum_{\xi \in \mathbb{Z}}F(x+\xi)e^{2\pi i (kt+ (m+n\xi)y+nz)}.
 \end{align*}
 Later in this section we will be applying the above decomposition to smooth functions. We would therefore like to know how to characterise the spaces $\mathcal{H}_{k,m,n}:= \mathcal{H}^{k,m,n}\cap C^{\infty}(\KT)$ and $\mathcal{H}_{k,l,m,0}:= \mathcal{H}^{k,l,m,0}\cap C^{\infty}(\KT)$. Clearly $\mathcal{H}_{k,l,m,0}= \mathcal{H}^{k,l,m,0}$ and the case of $\mathcal{H}_{k,m,n}$ is given in the following proposition.
 \begin{prop}\label{Sch}
 The space $\mathcal{H}_{k,m,n}:= \mathcal{H}^{k,m,n}\cap C^{\infty}(\KT)$ is given by
 $$\mathcal H_{k, m, n}=\left\{ \sum_{\xi\in \mathbb{Z}}F(x+\xi)e^{2\pi i(kt+(m+n\xi)y+nz)} \,\middle|\,F\in \mathcal{S}^2(\mathbb{R})\right\}, $$
 where $\mathcal{S}(\mathbb{R})$ is the space of Schwartzian functions
$$\mathcal{S}(\mathbb{R}) = \left\{F(x)\in C^\infty(\mathbb{R}) \,: \, \sup_{x\in\mathbb{R}} \abs{x^p \frac{d^q}{dx^q}F(x)}<\infty, \text{ for all } p,q \in \mathbb{N}  \right\}. $$
 \begin{proof}
  It is equivalent to show that the transform $W_{k,m,n}$, when restricted to $\mathcal{S}(\mathbb{R})$, yields a bijection  
  $$W_{k,m,n}:\mathcal{S}(\mathbb{R})\rightarrow \mathcal{H}_{k,m,n}. $$
  To see that this is the case, let $F(x) \in L^2(\mathbb{R})$ be a function such that $W_{k,m,n}(F) \in C^{\infty}(\KT)$. Then, note that the derivatives $\frac{\partial^q}{\partial x^q}W_{k,m,n}(F)$ each describe a Fourier series, with coefficients given by $\frac{d^q}{dx^q}
  F(x+\xi)$.
  The proposition then follows from the classical result that a function is smooth if and only if its Fourier coefficients tend to zero faster than any inverse polynomial. 
 \end{proof}
 \end{prop}

It is a simple matter to check that the spaces $\mathcal{H}_{k,m,n}$ and $\mathcal{H}_{k,l,m,0}$ are all invariant with respect to $\der{t}, \der{x}, \der{y}+x\der{z}$ and $\der{z}$.
This means that, if a linear PDE on $\KT$ is constructed using only these four derivatives, we can use the above decomposition to write the smooth solutions as a sum of functions which also solve the PDE. By Proposition \ref{propdecomp}, a smooth function is zero if and only if its components in $\mathcal{H}_{k,m,n}$ and $\mathcal{H}_{k,l,m,0}$ are all zero. Thus, in order to find functions $f$ and $g$ satisfying the two conditions \eqref{V1} and \eqref{V2}, it makes sense to consider solutions in $\mathcal{H}_{k,m,n}$ with $n\neq 0$ separately from solutions in $\mathcal{H}_{k,l,m,0}$. These two cases will be dealt with in the next two subsections.

\begin{rmk}
From a representation theory perspective, $L^2(\KT)$ corresponds to ind$_{\Gamma}^G(1)$ where $\Gamma=\mathbb Z\times H_3(\mathbb Z)$ and $G=\mathbb R\times H_3(\mathbb R)$, and each $\mathcal H^{k, m, n}$ corresponds to the representation $\rho_n$ of the Heisenberg group. As $m$ takes integer values between $0$ and $\abs{n}-1$, the function space $\mathcal H^{k,n} := \widehat{\bigoplus}_{m}\mathcal{H}^{k,m,n}$ corresponds to $|n|$ multiples of the irreducible unitary representation $\rho_n$ (see {\it e.g.}  \cite{Aus} and Theorem 1.109 in \cite{Fol}). This is generalised to the Howe-Richardson multiplicity formula for compact nilmanifolds. Moreover, each $\mathcal{H}_{k,l,m,0}$ corresponds to the irreducible unitary representation $\sigma_{lm}$ of the Heisenberg group in classical notation ({\it e.g.} Theorem 1.59 in \cite{Fol}). 
\end{rmk}

\begin{rmk}
Before diving into the detailed calculations, we remark that the above mentioned $L^2$ decomposition could be used to study almost complex theta function valued harmonic $(p, q)$-forms, {\it i.e.} the space $\mathcal H_{\bar\partial_E}^{p, q}(\KT, E)$ for a non-trivial complex vector bundle $E$ with a pseudoholomorphic structure over $\KT$ with a compatible almost Hermitian metric.  Similar to classical theta functions over an abelian variety, an element of $\mathcal H_{\bar\partial_E}^{p, q}(X, E)$ should be viewed as a vector valued function over the universal covering of $\KT$, {\it i.e.} $\mathbb R^4$, satisfying an elliptic system induced from the $\bar\partial_E$-harmonic form equations on $\KT$.
\end{rmk}

\subsection{Solving the $n\ne 0$ case with Laplace Integral Transforms}
Solutions in $\mathcal{H}_{k,m,n}$ with fixed $n\neq0$ and $0\leq m<\abs{n}$, take the form of
\begin{align}
    f = \sum_{\xi\in \mathbb{Z}}F_{k,m,n}(x+\xi)e^{2\pi i(kt+(m+n\xi)y+nz)},\label{odefKT}\\
    g = \sum_{\xi\in \mathbb{Z}}G_{k,m,n}(x+\xi)e^{2\pi i(kt+(m+n\xi)y+nz)}\label{odegKT}.
\end{align}
Plugging these into \eqref{V1} and \eqref{V2} then taking the Fourier expansion with respect to $t,y$ and $z$ we obtain an ODE system on the whole of $\mathbb{R}$.
\begin{align}
    \frac{d}{dx}
    \begin{pmatrix}
        F_{k,m,n}\\
        G_{k,m,n}
    \end{pmatrix}
    = (A_{n}x+B_{k,m,n})
    \begin{pmatrix}
        F_{k,m,n}\\
        G_{k,m,n}
    \end{pmatrix}\label{Main DE}
\end{align}
with
$$ A_{n} = 2\pi\begin{pmatrix}
        0 & n\\
        n & 0\\
        \end{pmatrix},$$
$$ B_{k,m,n} = 2\pi\begin{pmatrix}
        k                   & m-n\frac{a-i}{b}\\
        m-n\frac{a+i}{b}    & \frac{b}{4\pi}i-k\\
        \end{pmatrix}.$$
For every choice of $k,m$ and $n$, this ODE system has two independent smooth solutions. However, as we saw in Prop. \ref{Sch}, we require that $F_{k,m,n},G_{k,m,n} \in \mathcal{H}_{k,m,n}$ be Schwartz functions. Indeed, this Schwartzian condition must rule out the vast majority of smooth solutions, as the elliptic system given by \eqref{V1} and \eqref{V2} can only have a finite number of solutions. 

The coefficients of the ODE system \eqref{Main DE} are analytic in $\mathbb R$ (in fact, in $\mathbb C$), and it has an irregular singularity ({\it i.e.} essential singularity) of order two at infinity.  By standard ODE theory ({\it e.g.} Chapters $3$ and $5$ in \cite{cod}), there are two linearly independent analytic solutions of  \eqref{Main DE}. If we consider the fundamental matrices of the ODE systems at both positive and negative infinities,  they are of the form $e^{Q_0x^2+Q_1x}x^aP(x^{-1})$, where $P(x^{-1})$ is a formal power series in $x^{-1}$ and $Q_0$ is the diagonal matrix $diag(\pi n, -\pi n)$.

Hence, as $x\rightarrow +\infty$ in \eqref{Main DE} we have two independent local solutions, one that grows like $e^{|n|\pi x^{2}}$ and one that decays like $e^{-\abs{n}\pi x^{2}}$, and likewise for large negative $x$. If we have a single solution that decays in both directions then it must be Schwartzian, though we may instead have two independent solutions that both blow up at one end while decaying at the other. In theory, both of these cases are possible. If $B_{k,m,n}=0$ then clearly we would have a Schwartzian solution. However, since the elliptic system \eqref{V1} and \eqref{V2} cannot have infinitely many solutions, for nearly all values of $k,m$ and $n$, no Schwartzian solution exists. 

A complete description of when each of these two cases occur is given below in a more general setting, which should have an independent interest in the ODE theory.  
\begin{thm}\label{1stode}
Let $A,B \in M_{2}(\mathbb{C})$ be matrices and let $A$ have two distinct, real eigenvalues $\lambda_{1}$, $\lambda_{2}$ with $\lambda_{1}>0>\lambda_2$ then the equation
\begin{align}
    \frac{d}{dx}
    \begin{pmatrix}
        f\\
        g
    \end{pmatrix}
    = (Ax+B)
    \begin{pmatrix}
        f\\
        g
    \end{pmatrix}\label{Ax+B}
\end{align}
has a pair of solutions $f,g\in L^{2}(\mathbb{R})$ if and only if the following holds: Given $P\in GL(2, \mathbb C)$ such that $PAP^{-1}$ is diagonal and writing $PBP^{-1}$ as $\begin{pmatrix}
    b_{1} & b_{2}\\
    b_{3} & b_{4}
\end{pmatrix}$ we have $b_{2}b_{3}\in (\lambda_{1}-\lambda_{2})\cdot \mathbb{Z}^-$, and in this situation both $f$ and $g$ are Schwartz functions.
\end{thm}
Clearly if $\lambda_{1}$ and $\lambda_{2}$ are both positive then all pairs of solutions $f,g$ will blow up in both the positive and negative directions, while if they are both negative all pairs $f,g$ will decay in both directions. Note also that the arguments we use below still apply when $A$ has two complex eigenvalues $\lambda_1, \lambda_2$ with $\Re\lambda_1>0>\Re\lambda_{2}$. Here however we will restrict our attention to the real situation to simplify the notation and also because it is sufficient for all our applications in this paper.

\begin{proof}
If we write down the second order ODE satisfied by $f$ or $g$, the coefficients will involve a third order polynomial of $x$, and there is no efficient method known to study these types of equations. The trick is here is to simplify the above equation \eqref{Ax+B} slightly by left-multiplying the solution by $P$ and adding an $e^{-\frac{1}{2}\lambda_{2}x^{2}}$ term inside the derivative. 
This replaces $A$ with a matrix with only one non-zero entry, such that our equation becomes
\begin{align} \label{diagode}
    \frac{d}{dx}
    \begin{pmatrix}
        \psi\\
        \phi
    \end{pmatrix}
    = \left(\begin{pmatrix}
        \lambda_{1}-\lambda_{2} & 0\\
        0 & 0
    \end{pmatrix}x+PBP^{-1}\right)
    \begin{pmatrix}
        \psi\\
        \phi
    \end{pmatrix},
\end{align}
where
$$    \begin{pmatrix}
        \psi\\
        \phi
    \end{pmatrix}=  e^{-\frac{1}{2}\lambda_{2}x^{2}}P
    \begin{pmatrix}
        f\\
        g
    \end{pmatrix}.$$
From this we can show either $\phi$ or $\psi$ must satisfy a second order ODE, both of which can be solved using a Laplace integral transform:
\begin{align}
    &\psi''-((\lambda_{1}-\lambda_{2})x+b_{1}+b_{4})\psi'+((\lambda_{1}-\lambda_{2})b_{4}x+b_{1}b_{4}-b_{2}b_{3}-(\lambda_{1}-\lambda_{2}))\psi=0, \label{psi}\\
   &\phi''-((\lambda_{1}-\lambda_{2})x+b_{1}+b_{4})\phi'+((\lambda_{1}-\lambda_{2})b_{4}x+b_{1}b_{4}-b_{2}b_{3})\phi=0. \label{phi}
\end{align}

As detailed in \cite{cod}, in order to find a function $h$ that satisfies
$$(p_{2}x+q_{2})h''+(p_{1}x+q_{1})h'+(p_{0}x+q_{0})h=0, $$
we can write $h$ as
$$h(x) = \int_{C}\varphi(s) e^{sx}ds $$
where $C$ is some contour in the complex plane $\mathbb{C}$. Then, defining
\begin{align*}
    P(s) = p_{2}s^{2}+p_{1}s+p_{0}\\
    Q(s) = q_{2}s^{2}+q_{1}s+q_{0}
\end{align*}
and choosing $C$ so that
$$V(s) = \exp\left(\int^{s} \frac{Q(\sigma)}{P(\sigma)}d\sigma\right)e^{s x} $$
takes the same value at both (possibly infinite) endpoints for all $x$ when $s$ parameterises the contour $C$, we can find a solution
$$\varphi(s) = \frac{1}{P(s)}\exp\left(\int^{s} \frac{Q(\sigma)}{P(\sigma)}d\sigma\right) \quad \mathrm{\&}\quad h(x)=\int_C\frac{V(s)}{P(s)}ds. $$

In our specific case, first solving \eqref{phi} for $\phi$ we find that
$$P_{\phi}(s) = (\lambda_{1}-\lambda_{2})(b_{4}-s), \quad \quad Q_{\phi}(s) = s^{2}-(b_{1}+b_{4})s+(b_{1}b_{4}-b_{2}b_{3}),$$
which gives us the solution
\begin{align}
    \phi(x) = \frac{1}{\lambda_2-\lambda_1}\int_{C}(s-b_{4})^{\frac{b_{2}b_{3}}{\lambda_{1}-\lambda_{2}}-1}\exp(-\frac{1}{\lambda_{1}-\lambda_{2}}\left(\frac{s^2}{2}-b_{1}s\right)+xs) ds,\label{phiint}
\end{align}
with 
$$V_{\phi}(s) = (s-b_{4})^{\frac{b_{2}b_{3}}{\lambda_{1}-\lambda_{2}}}\exp(-\frac{1}{\lambda_{1}-\lambda_{2}}\left(\frac{s^2}{2}-b_{1}s\right)+xs). $$
The function $V_{\phi}$ tends to zero as $s$ grows large within the shaded regions. 

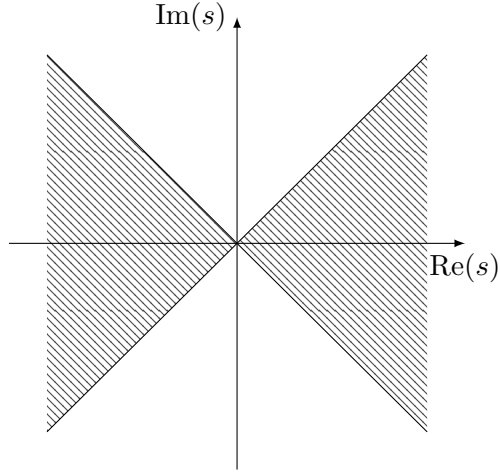
\begin{center}
\begin{tikzpicture}
\fill[pattern color=black!70,scale=.5,domain=.2:5.5,smooth,pattern=north west lines] (-5,5) -- plot({\x},-{\x}) -- (5,-5) -- (5,5) -- plot ({\x},{\x}) -- (-5,-5) -- cycle;
    \draw[scale=.5,domain=-5:5,smooth] plot ({\x},{\x});
    \draw[scale=.5,domain=-5:5,smooth] plot ({\x},{-\x});
    \draw[>=latex,->] (-3,0) -- (3,0) node[below] {$\Re(s)$};
    \draw[>=latex,->] (0,-3) -- (0,3) node[left] {$\Im(s)$};
\end{tikzpicture}
\end{center}

This means the contours $C_{1}$ and $C_{2}$ starting and ending on the left, resp. right, and encircling the point $s = b_{4}$ are valid choices for $C$. If we require that $\frac{b_{2}b_{3}}{\lambda_{1}-\lambda_{2}}\notin\mathbb{Z}$ then integrating along these contours would give two independent solutions $\phi_{1}$ and $\phi_{2}$ of \eqref{diagode}.
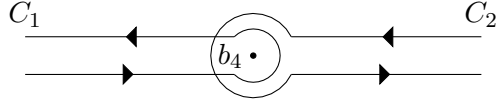
\begin{center}
\begin{tikzpicture}    
\draw[fill] (0,0) circle (1pt) node[left] {$b_{4}$};
\begin{scope}[every node/.style={sloped,allow upside down}]
    \draw (-.25,.25) -- node {\midarrow} (-3,.25) node[above] {$C_{1}$};
    \draw (-3,-.25) -- node {\midarrow} (-.25,-.25);
    \draw (3,.25)  node[above] {$C_{2}$} -- node {\midarrow} (.5,.25);
    \draw (.5,-.25) -- node {\midarrow} (3,-.25);
\end{scope}
    \draw [domain=-135:135] plot ({0.3535*cos(\x)}, {0.3535*sin(\x)}); 
    \draw [domain=26.45:333.55] plot ({0.559*cos(\x)}, {0.559*sin(\x)});
\end{tikzpicture}
\end{center}
This is in essence the same as in the treatment for the physicists' Hermite equation from the mathematical appendices of \cite{lan}, and as in \S a there we will use a substitution
to explore the behaviour of solutions as $x\rightarrow \pm \infty$. 

Let
$t := s-(b_{1}+(\lambda_{1}-\lambda_{2})x)$ and let $\tilde C$ be the new contour transformed from $C$. 
Then our expression for $\phi$ becomes
\begin{align*}
    \phi(x) =\frac{1}{\lambda_2-\lambda_1} \exp(\frac{(b_1+(\lambda_1-\lambda_2)x)^2}{2(\lambda_1-\lambda_2)})\int_{\tilde{C}}(t+b_{1}-b_{4}+(\lambda_{1}-\lambda_{2})x)^{\frac{b_{2}b_{3}}{\lambda_{1}-\lambda_{2}}-1}e^{-\frac{t^{2}}{2(\lambda_{1}-\lambda_{2})}} dt.
\end{align*}
Recall that 
$$   e^{\frac{1}{2}\lambda_{2}x^{2}} \begin{pmatrix}
        \psi\\
        \phi
    \end{pmatrix}=  P
    \begin{pmatrix}
        f\\
        g
    \end{pmatrix}.$$
Since $L^{2}(\mathbb{R})$ is closed under addition, given any invertible matrix $P\in GL(2, \mathbb{C})$ we can say that $\begin{pmatrix}
        f\\
        g
    \end{pmatrix}$ is a pair of square-integrable functions if and only if $P\begin{pmatrix}
        f\\
        g
    \end{pmatrix}$ is a pair of square-integrable functions. In particular, we will have proven the ``only if" part of the theorem if we can show no linear combination of $e^{\frac{1}{2}\lambda_{2}x^{2}}\phi_{1}$ and $e^{\frac{1}{2}\lambda_{2}x^{2}}\phi_{2}$ can be square-integrable other than the cases specified.

    Use $\tilde{C}_{1}$ and $\tilde{C}_{2}$ to denote the new contours transformed from $C_1$ and $C_2$ after substitution. We can see that as $x\rightarrow +\infty$, $\tilde{C}_{1}$ will shift to the left causing the integral along it to decay like $e^{-\frac{1}{2}(\lambda_{1}-\lambda_{2})x^{2}}$ and hence $e^{\frac{1}{2}\lambda_{2}x^{2}}\phi_{1}$ will decay at the rate of $e^{\frac{1}{2}\lambda_{2}x^{2}}$. $\tilde{C}_{2}$ will extend the whole horizontal direction and  the integral does not tend to zero, meaning $e^{\frac{1}{2}\lambda_{2}x^{2}}\phi_{2}$ grows like $e^{\frac{1}{2}\lambda_{1}x^{2}}$.
    
As $x\rightarrow -\infty$ our contours shift to the right instead. This results in $e^{\frac{1}{2}\lambda_{2}x^{2}}\phi_{1}$ now being the one to grow like $e^{\frac{1}{2}\lambda_{1}x^{2}}$ and $e^{\frac{1}{2}\lambda_{2}x^{2}}\phi_{2}$ the one decaying like $e^{\frac{1}{2}\lambda_{2}x^{2}}$. Clearly this means any linear combination of these two functions will blow up at either $\infty, -\infty$ or both.

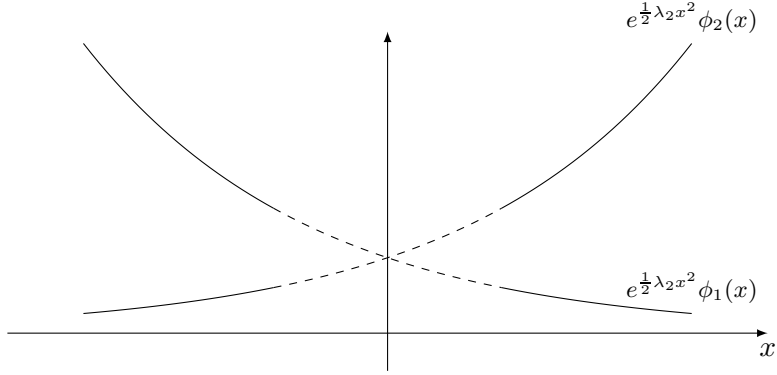
\begin{center}
    \begin{tikzpicture}
        \draw [domain=-4:-1.5] plot ({\x},{1.4^(\x)}) ;
        \draw [domain=-4:-1.5] plot ({\x},{1.4^(-\x)}) ; 
        \draw [domain=-1.5:1.5,dashed]plot ({\x},{1.4^(\x)}) ;
        \draw [domain=-1.5:1.5,dashed] plot ({\x},{1.4^(-\x)}) ;    
        \draw [domain=1.5:4]plot ({\x},{1.4^(\x)}) node[above] {\footnotesize{$e^{\frac{1}{2}\lambda_{2}x^{2}}\phi_{2}(x)$}};
        \draw [domain=1.5:4] plot ({\x},{1.4^(-\x)}) node[above] {\footnotesize{$e^{\frac{1}{2}\lambda_{2}x^{2}}\phi_{1}(x)$}};
        \draw[>=latex,->] (-5,0) -- (5,0) node[below] {$x$};
        \draw[>=latex,->] (0,-.5) -- (0,4);
    \end{tikzpicture}
\end{center}

If $\frac{b_{2}b_{3}}{\lambda_{1}-\lambda_{2}}\in\mathbb{Z}^-\cup\{0\}$, then the integrals along the horizontal directions of the path of integration cancel, and the two integrals along $\tilde C_1$ and $\tilde C_2$ reduce to an integral  along a loop around $t=b_4-b_1-(\lambda_1-\lambda_2)x$. This gives rise to a solution $\phi(x)$ which grows at most as $e^{Kx}$ at both ends (when $b_4=0$, it is essentially an Hermite polynomial). Hence $e^{\frac{1}{2}\lambda_{2}x^{2}}\phi$ decays as $e^{\frac{1}{2}\lambda_{2}x^{2}}$ at both ends.

The same argument applies to \eqref{psi} for $\psi$ when $\frac{b_{2}b_{3}}{\lambda_{1}-\lambda_{2}}+1\in\mathbb{Z}^-\cup\{0\}$, {\it i.e.} $\frac{b_{2}b_{3}}{\lambda_{1}-\lambda_{2}}\in\mathbb{Z}^-$. In this case, $e^{\frac{1}{2}\lambda_{2}x^{2}}\psi$ decays as $e^{\frac{1}{2}\lambda_{2}x^{2}}$ at both ends. As in particular $b_2\ne 0$, this Schwartz function $\psi$ gives rise to another function $\phi$ using the first relation of \eqref{diagode}. It turns out this $\phi$  solves the other relation of \eqref{diagode} and also \eqref{phi}, in fact it is equal to the $\phi$ in the previous paragraph, up to multiplication by a constant. This pair $(\phi, \psi)$ gives rise to a solution of  \eqref{Ax+B}, which is  a vector valued Schwartz function when $\frac{b_{2}b_{3}}{\lambda_{1}-\lambda_{2}}\in\mathbb{Z}^-$. We remark that this is the only $L^2(\mathbb R^2)$ solution of \eqref{Ax+B} as by ODE theory there is only one solution decays as $e^{\lambda_2x^2}$ at $+\infty$ (or $-\infty$).

Lastly, we study the solution of \eqref{psi} when $\frac{b_{2}b_{3}}{\lambda_{1}-\lambda_{2}}+1\in\mathbb{Z}^+$. Contours $C_{1}$ and $C_{2}$ both give the trivial solution. Instead we define the contours $C_{3}$ and $C_{4}$ to be the lines parallel to the real axis, running from $b_{4}$ to $-\infty$ and $+\infty$ respectively. These satisfy the condition that $V_{\psi}(s) = (s-b_{4})^{\frac{b_{2}b_{3}}{\lambda_{1}-\lambda_{2}}+1}\exp(-\frac{1}{\lambda_{1}-\lambda_{2}}\left(\frac{s^2}{2}-b_{1}s\right)+xs) = 0$ at both endpoints and so give rise to two independent solutions $\psi_{3}$ and $\psi_{4}$ of \eqref{psi}. We have (as in \eqref{phiint} for $\phi$) \begin{align*}
    \psi_{3,4}(x) = \frac{1}{\lambda_2-\lambda_1}\int_{C_{3,4}}(s-b_{4})^{\frac{b_{2}b_{3}}{\lambda_{1}-\lambda_{2}}}\exp(-\frac{1}{\lambda_{1}-\lambda_{2}}\left(\frac{s^2}{2}-b_{1}s\right)+xs) ds.
\end{align*}
Then we can use exactly the same argument as we did for the case when $\frac{b_{2}b_{3}}{\lambda_{1}-\lambda_{2}}\notin\mathbb{Z}$ to see that $e^{\frac{1}{2}\lambda_{2}x^{2}}\psi_{3}$ decays like $e^{\frac{1}{2}\lambda_{2}x^{2}}$ as $x\rightarrow +\infty$ and grows like $e^{\frac{1}{2}\lambda_{1}x^{2}}$ as $x\rightarrow -\infty$ whilst the opposite is true of $e^{\frac{1}{2}\lambda_{2}x^{2}}\psi_{4}$. This means when $\frac{b_{2}b_{3}}{\lambda_{1}-\lambda_{2}}+1\in\mathbb{Z}^+$ all solutions blow up in either the positive or negative directions; hence, there are no $L^{2}$ solutions.
\end{proof}

It should be noted that in this last case we could write our solutions explicitly. First by directly calculating $\psi_{3}$ and $\psi_{4}$ when $\frac{b_{2}b_{3}}{\lambda_{1}-\lambda_{2}}=0$ and $1$, using the complementary error function
$$\erfc(x) := \frac{2}{\sqrt{\pi}}\int_{x}^{\infty}e^{-t^{2}}dt  $$
Then, using integration by parts, we could find a recurrence relation allowing us to write solutions for all other positive integer values of $\frac{b_{2}b_{3}}{\lambda_{1}-\lambda_{2}}$ in terms of these first two.

\begin{rmk}
In the preceding theorem we considered solutions as functions of a real variable $x$. If instead we allow $x$ to take complex values we see that our study of asymptotic behaviour as $x \rightarrow \pm \infty$ is really just a restriction of the lateral connection problem to $\mathbb{R}$. It would be interesting to see if a full description of the Stokes phenomenon for this linear system could be obtained using a similar method.
\end{rmk}

We apply Theorem \ref{1stode} to equation \eqref{Main DE}. 
When $n>0$,  we have $\lambda_{1} = 2\pi n, \lambda_{2} = -2\pi n$ and
$$PB_{k, m, n}P^{-1} = 2\pi\begin{pmatrix}
      m-\frac{na}{b} +\frac{b}{8\pi}i & k-\frac nb i-\frac{b}{8\pi}i\\
    k+\frac{n}{b}i-\frac{b}{8\pi}i & -m+\frac{na}{b}+\frac{b}{8\pi}i
\end{pmatrix}, $$
where $P =\frac{1}{\sqrt{2}}\begin{pmatrix}
    1 & 1\\
    1 & -1
\end{pmatrix}$. 
In order for us to have a pair of solutions $F_{k,m,n}, G_{k,m,n}\in \mathcal{S}(\mathbb{R})$ we must have 

$$4\frac{\pi^2}{n}\left(k-\frac nb i-\frac{b}{8\pi}i\right)\left(k+\frac{n}{b}i-\frac{b}{8\pi}i\right)\in 4\pi \mathbb{Z}^-. $$
The imaginary part of the left hand side is $-\frac{kb\pi}{n}i$, so we can only have solutions when $k$ is zero. Then looking at the real part and setting $k=0$, we also need $b$ to satisfy
$$\frac{\pi}{n}\left(\left(\frac{n}{b}\right)^{2}-\left(\frac{b}{8\pi}\right)^{2}\right)\in  \mathbb{Z}^-.$$
That is to say we have a Schwartzian solution to \eqref{Main DE} whenever there is some $u\in\mathbb{Z}^-$ such that $b$ is a solution to  
\begin{equation} b^{4}+64\pi  nu b^{2}-64\pi^{2}n^{2}=0 \label{b eq}, \end{equation}
or in terms of $d=\frac{b}{8\pi}$,
\begin{equation}64\pi^2d^4-64\pi und^2- n^2=0. \label{d eq}\end{equation}
For \eqref{d eq} to hold requires $8\pi d^2\in \mathbb Z[\sqrt{D}]$ for some integer $D>0$. 
For example, since $\pi$ is a transcendental number, no rational number $d=\frac{p}{q}$ can satisfy equation \eqref{d eq} for any choice of $u, n\in \mathbb Z$.

When $n<0$, in all the above relations we have $|n|$ instead of $n$, and we have essentially the same discussion. 

\subsection{Solving the $n=0$ case}\label{n=0NT}
In the case when $n=0$, we want solutions in $\mathcal{H}_{k,l,m,0}$ with the form of 
\begin{align*}
    f = F_{k,l,m,0}e^{2\pi i (kt+lx+my)},\\
     g = G_{k,l,m,0}e^{2\pi i (kt+lx+my)}.
\end{align*}
Plugging these into \eqref{V1} and \eqref{V2} then taking the Fourier expansion with respect to $t,x$ and $y$, this tells us we have a solution when $F_{k,l,m,0}$ and $G_{k,l,m,0}$ satisfy
\begin{align*}
    -mF_{k,l,m,0}+\left(k+il-\frac{bi}{4\pi}\right)G_{k,l,m,0}&=0,\\
    (k-il)F_{k,l,m,0}+mG_{k,l,m,0}&=0.
\end{align*}

If $m=0$ then our first equation tells us that we either have $G_{k,l,0,0}=0$ or we have $k=0$ and $b=4\pi l$. Our second equation tells us either $F_{k,l,0,0}=0$ or $k=l=0$.
So we have a family of solutions given by
\begin{equation}f=C_{0},\quad g=0, \label{sol 1}\end{equation}
and another family of solutions when $b = 4\pi l \in 4\pi\mathbb{Z}\backslash\{0\}$, given by
\begin{equation} f=0,\quad g=C_{1}e^{2\pi i lx}. \label{sol 2}\end{equation}

If instead we take $m\neq 0$ then we can rewrite our equations as 
$$ \left(k^{2}+l^{2}+m^{2}-\frac{b}{4\pi}l-\frac{b}{4\pi}ik\right)F_{k,l,m,0}=0,$$
$$ G_{k,l,m,0}=-\frac{k-il}{m} F_{k,l,m,0}.$$
Clearly, if we want a nontrivial solution we need to find where
$$ k^{2}+l^{2}+m^{2}-\frac{b}{4\pi}l-\frac{b}{4\pi}ik=0.$$
This is the case exactly when $k=0$ and nonzero $l,m$ are chosen such that $b=4\pi (l^{2}+m^{2})/l$. This yields the solutions
\begin{equation}f=mC_{2}e^{2\pi i(lx+my)},\quad g = ilC_{2}e^{2\pi i(lx+my)}. \label{sol 3}\end{equation}
Note the solution \eqref{sol 2} is included in the family of solutions \eqref{sol 3}, but solution \eqref{sol 1} is not.

\section{Counting the size of $h^{0,1}$}
We will finish the computation of $h^{0,1}$ for $J_{a,b}$ with the standard orthonormal metric in this section.

Suppose we choose an almost complex structure $J_{a,b}$ such that $b$ does not solve \eqref{b eq}. How many independent solutions does \eqref{Main DE} actually have? Counting the solutions provided by \eqref{sol 1}-\eqref{sol 3} is equivalent to 
asking how many $l$ and $m$ satisfy
$$bl=4\pi(l^{2}+m^{2}),$$ 
which is equivalent to the number theoretic question: how many pairs of integers $(m,l)$ satisfy
\begin{equation}\label{circlat}
(l-d)^{2}+m^{2}=d^{2},
\end{equation}
if we relabel $\frac {b}{8\pi}$ as $d$. Notice the pair $(l, m)=(0, 0)$ corresponds to the trivial solution $s=0$ of \eqref{V1} and \eqref{V2}, but in the counting we can view it as corresponding to 
the solution \eqref{sol 1}. Other pairs satisfying \eqref{circlat} correspond to the solutions in \eqref{sol 3} which include solutions from \eqref{sol 2}. Apparently, \eqref{circlat} has no solutions beyond \eqref{sol 1} except for when $d$ is rational.  

Counting the number of solutions can be thought of as asking how many lattice points lie on a circle with centre $(d,0)$ and radius $d$. For instance, when $d= \frac 52$ we have 6 solutions as shown below. 

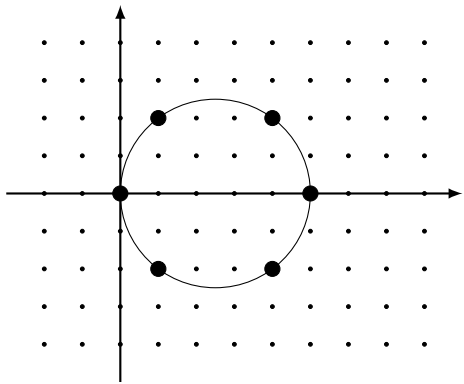
\vspace{5mm}
\begin{center}
\begin{tikzpicture}
    [
        dot/.style={circle,draw=black, fill,inner sep=.5pt},
    ]

\foreach \x in {-1,-0.5,...,4}{
    \foreach \y in {-2,-1.5,...,2}{
        \node[dot] at (\x,\y){ };
}}

 \node[dot, inner sep = 2pt] at (0,0){ };
 \node[dot, inner sep = 2pt] at (2.5,0){ };
 \node[dot, inner sep = 2pt] at (0.5,1){ } ;
 \node[dot, inner sep = 2pt] at (2,1){ };
 \node[dot, inner sep = 2pt] at (0.5,-1){ } ;
 \node[dot, inner sep = 2pt] at (2,-1){ };

\draw (1.25,0) circle (1.25cm);

\draw[->,thick,-latex] (0,-2.5) -- (0,2.5);
\draw[->,thick,-latex] (-1.5,0) -- (4.5,0);

\end{tikzpicture}
\end{center}
\vspace{5mm}

When $d$ is an integer this problem is very well understood and the number of such integer pairs is denoted $r_2(d^2)$, see for instance \cite{HW}. First we write $d^{2}$ as a unique product of prime numbers 
$$ d^{2} = 2^{\alpha_{0}}p_1^{\alpha_{1}}\dots p_s^{\alpha_{s}}q_1^{\beta_{1}}\dots q_t^{\beta_{t}},$$
where $p_{i}\equiv 3 \mod 4$ for all $i$ and $q_{j}\equiv 1 \mod 4$ for all $j$. The number of solutions is then given by 
$$h^{0,1} = 4(\beta_{1}+1)(\beta_{2}+1)\dots (\beta_{t}+1). $$
This reveals the interesting fact that by changing our choice of $b$ we can make $h^{0,1}$ become arbitrarily large.
It should be noted that if any of the powers of the $p_{i}$'s were odd then we would not have any solutions, but since we are looking at a square number the powers are guaranteed to be even. 

Moreover, when $d=\frac{p}{q}$ with $\gcd(p, q)=1$ and $q$ is small, we can also compute the number of solutions.

\begin{thm}\label{h01}
For the family of almost complex structures $J_{a, b}$, $b\neq 0$, and the standard almost K\"ahler metrics on the Kodaira-Thurston manifold, whenever $d:= \frac{b}{8\pi} \in \mathbb{Q}$, the Hodge number $h^{0,1}$ is equal to the number of integer pairs $(l,m)$ solving the generalised Gauss circle problem \eqref{circlat}.

Furthermore, if $d =  \frac pq$, with $\gcd(p,q)=1$ and $q \le 5$, we have 
$$
h^{0,1}=\begin{cases} \begin{array}{cll}
4(\beta_{1}+1)(\beta_{2}+1)\dots (\beta_{t}+1)& \hbox{ if $q=1$},\\
2(\beta_{1}+1)(\beta_{2}+1)\dots (\beta_{t}+1)& \hbox{ if $q=2$},\\
(\beta_{1}+1)(\beta_{2}+1)\dots (\beta_{t}+1)& \hbox{ if $q=3$},\\
(\beta_{1}+1)(\beta_{2}+1)\dots (\beta_{t}+1)& \hbox{ if $q=4$},\\
(\beta_{1}+1)(\beta_{2}+1)\dots (\beta_{t}+1)& \hbox{ if $q=5$}.
\end{array}
\end{cases}
$$
where $p^2$ has the prime factorisation $p^2= 2^{\alpha_{0}}p_1^{\alpha_{1}}\dots p_s^{\alpha_{s}}q_1^{\beta_{1}}\dots q_t^{\beta_{t}}$ with $p_{i}\equiv 3 \mod 4$ for all $i$ and $q_{j}\equiv 1 \mod 4$ for all $j$.
\end{thm}
\begin{proof}
First, any rational number $d=\frac{p}{q}$ cannot solve \eqref{d eq}. Hence, all solutions to \eqref{V1} and \eqref{V2} are provided by linear combinations of the ones with $n=0$ in the decomposition \eqref{decomp} from \eqref{sol 1} and \eqref{sol 3}, whose dimension is equal to the number of lattice points of \eqref{circlat}. In the following, we compute this number. 

In the proof, we always write $q(l-d)=ql'-d'$ where $l'=l-\floor{\frac{p}{q}}$ and $d'=q\{\frac{p}{q}\}$. By abusing notation, we usually write $l$ for $l'$ in the following.

The case of $q=1$ is solved above. 

When $q=2$, then $p$ is odd. We can rewrite \eqref{circlat} as $(2l-1)^2+(2m)^2=p^2$. For any integer solution $(x, y)$ of $x^2+y^2=p^2$, one and only one from $(x, y)$ and $(y,x)$ is of the type $(2l-1, 2m)$. Thus, $h^{0,1}$ is half of the number of lattice points on $x^2+y^2=p^2$.

When $q=3$, then $p$ is not divisible by $3$. Rewrite \eqref{circlat} as $(3l-d')^2+(3m)^2=p^2$, where $d'$ is $1$ or $2$. For any integer solution $(x, y)$ of $x^2+y^2=p^2$, one and only one among $(x, y)$, $(x, -y)$, $(y, x)$ and $(-y, x)$ is of the type $(3l-d', 3m)$ for a given $d'$. Thus $h^{0,1}$ is a quarter of the number of lattice points on $x^2+y^2=p^2$.

When $q=4$, then $p\equiv 1$ or $3$ (mod $4$). Rewrite \eqref{circlat} as $(4l-d')^2+(4m)^2=p^2$, where $d'$ is $1$ or $3$. We look at the equation $x^2+y^2=p^2$ modulo $8$, then the even term has to be a multiple of $4$. Hence, for any integer solution $(x, y)$ of $x^2+y^2=p^2$, one and only one among $(x, y)$, $(x, -y)$, $(y, x)$ and $(-y, x)$ is of the type $(4l-d', 4m)$ for a given $d'=1$ or $3$. Thus $h^{0,1}$ is a quarter of the number of lattice points on $x^2+y^2=p^2$.

When $q=5$, then $p$ is not divisible by $5$. Rewrite \eqref{circlat} as $(5l-d')^2+(5m)^2=p^2$, where $d'$ is $1, 4$ or $2, 3$. We look at the equation $x^2+y^2=p^2$ modulo $5$, the left hand side is $1$ mod $5$ if $d'=1, 4$, or is $4$ mod $5$ if $d'=2, 3$. In both cases, for any integer solution $(x, y)$ of $x^2+y^2=p^2$, one and only one among $(x, y)$, $(x, -y)$, $(y, x)$ and $(-y, x)$ is of the type $(5l-d', 5m)$ for a given $d'$. Thus $h^{0,1}$ is a quarter of the number of lattice points on $x^2+y^2=p^2$.
\end{proof}

The above argument cannot continue for $q\ge 6$ as we have $4^2+3^2=5^2+0^2=5^2$. It would be interesting to know in general how many integer solutions of \eqref{circlat} there are.

\begin{cor}\label{anyint}
For any nonnegative integer $n=4K, 2K$ or $K$ where $K$ is odd, there is an almost complex structure that is compatible with its standard orthonormal metric on $\KT$ whose $h^{0,1}=n.$
\end{cor}
\begin{proof}
When $K=1$, we take $b=8\pi, 4\pi, 2\pi$ respectively. 

When $K>1$, we take $b=\frac{8\pi\cdot 5^{\frac{K-1}{2}}}{q}$ where $q=1, 2, 3$ respectively. These are Schinzel circles \cite{Schin}.
\end{proof}

We notice that for the vast majority of members of the family of almost complex structures $J_{a, b}$, we have $h^{0,1}=1$ as this holds for any irrational $d=\frac{b}{8\pi}$ which does not solve \eqref{d eq} (in particular, those with $[\mathbb Q(\pi d^2):\mathbb Q]>2$) and an arbitrary $a$. On the other hand, we can compute $h^{0,1}$ for those $d$  that do solve \eqref{d eq}.
\begin{prop}\label{sol>0}
If some $d$ (with $8\pi d^2\in \mathbb Z[\sqrt{D}]$ for some $D\in \mathbb Z^+$) solves \eqref{d eq} for a given $n\in \mathbb Z\setminus \{0\}$ and a certain $u\in \mathbb Z^-$, then $h^{0,1}=2|n|+1$ for  the almost complex structure $J_{a, 8\pi d}$, $\forall a\in \mathbb R$, with its standard orthonormal metric on $\KT$.
\end{prop}
\begin{proof}
Notice for $n\ne 0$, $\pm n$ gives the same equation \eqref{d eq} to solve where $n$ is replaced by $|n|$. Hence, without loss, we can assume $n>0$.

For any $d$, there is only one $n> 0$ that could solve \eqref{d eq}. If there is another $N> 0$ and $U\in\mathbb Z^-$ solving \eqref{d eq} for $d$, then $$n(32u+\sqrt{(32u)^2+1})=N(32U+\sqrt{(32U)^2+1}).$$ This holds only when $n=N$ and $u=U$.

Hence, by Theorem \ref{1stode}, for each integer $0\le m<n$, there will be a Schwartzian solution to \eqref{Main DE}. There is at most one for each $m$, as by ODE theory only one solution decays as $e^{-nx^2}$ at $+\infty$. 
Similarly, there will be $n$ Schwartzian solutions when we start with $-n$.

Moreover, we have one and only one solution contributed by $(l,m)=(0,0)$ in the $n=0$ case as $d$ is irrational. In other words, only solution \eqref{sol 1} will contribute. 

In total, we have $2|n|+1$ dimensions of solutions to \eqref{V1} and \eqref{V2}. This implies $h^{0,1}=2|n|+1$.
\end{proof}

In particular, Corollary \ref{anyint} and Proposition \ref{sol>0} implies Theorem \ref{intro1} for the family $J_b=J_{0,b}$. 

\section{The Kodaira-Spencer problem}\label{Sec KS}
In this section, we will choose metrics that are not standard orthonormal and give a negative answer to Question \ref{KS} for $h^{0,1}$ on the Kodaira-Thurston manifold combining the computation done in previous sections.

We choose an almost Hermitian metric, say $h_{\rho}$, such that $\phi_1-\rho\phi_2$ and $\phi_2$ form a unitary basis of the holomorphic tangent bundle, where $\rho$ is a real number which will be specified later. Then for a general $(0,1)$ form $s=f\bar \phi_1+g\bar\phi_2=f(\bar\phi_1-\rho\bar\phi_2)+(g+\rho f)\bar\phi_2$, we still have equation \eqref{V1} but equation \eqref{V2} would become \begin{equation}((1+\rho^2)V_1+\rho V_2)(f)+(V_2+\rho V_1)(g)=0,\label{V2KS}\end{equation}
since
\begin{align*}
    \partial*(f\bar{\phi}_{1}+g\bar{\phi}_{2}) &= \partial(f\phi_{2}\wedge(\bar{\phi}_{1}-\rho\bar\phi_2)\wedge\bar{\phi}_{2}-(g+\rho f)(\phi_{1}-\rho\phi_2)\wedge(\bar{\phi}_{1}-\rho\bar\phi_2)\wedge\bar{\phi}_{2})\\
   &= \partial((f+\rho^2f+g\rho)\phi_{2}\wedge\bar{\phi}_{1}\wedge\bar{\phi}_{2}-(g+\rho f)\phi_{1}\wedge\bar{\phi}_{1}\wedge\bar{\phi}_{2}) \\
    &= \left((1+\rho^2)V_{1}(f)+\rho V_1(g)+V_{2}(g)+\rho V_2(f)\right)\phi_{1}\wedge\phi_{2}\wedge\bar{\phi}_{1}\wedge\bar{\phi}_{2}\\
    &= \left(((1+\rho^2)V_1+\rho V_2)(f)+(V_2+\rho V_1)(g)\right)\phi_{1}\wedge\phi_{2}\wedge\bar{\phi}_{1}\wedge\bar{\phi}_{2}=0.
\end{align*} 

We still apply the decomposition \eqref{decomp} and first look for solutions for a fixed $n\neq0$, in the form of \eqref{odefKT} and \eqref{odegKT}.

Plugging these into \eqref{V1} and \eqref{V2KS} and after some rearrangement we obtain the ODE of type \eqref{Main DE} with
$$ A_{n} = 2\pi\begin{pmatrix}
        0 & \frac{n}{1+\rho^2}\\
        n & 0\\
        \end{pmatrix},$$
$$ B_{k,m,n} = 2\pi\begin{pmatrix}
        k + \frac{2\rho ni}{(1+\rho^2)b}    & \frac{1}{1+\rho^2}(m-n\frac{a-i}{b}+\frac{\rho bi}{4\pi})\\
        m-n\frac{a+i}{b}    & \frac{b}{4\pi}i-k\\
        \end{pmatrix}.$$

In the setting of Theorem \ref{1stode}, we have $\lambda_{1} = \frac{2\pi |n|}{\sqrt{1+\rho^2}}, \lambda_{2} = -\frac{2\pi |n|}{\sqrt{1+\rho^2}}$ and
$$PB_{k, m, n}P^{-1} = \frac{\pi}{\sqrt{1+\rho^2}}\begin{pmatrix}
      b_1 & b_2\\
   b_3 & b_4
\end{pmatrix}, $$
where $P =\begin{pmatrix}
    1 & 1\\
    \sqrt{1+\rho^2} & -\sqrt{1+\rho^2}
\end{pmatrix}$,  $P^{-1} =\frac{1}{2\sqrt{1+\rho^2}}\begin{pmatrix}
    \sqrt{1+\rho^2} & 1\\
    \sqrt{1+\rho^2} & -1
\end{pmatrix}$, and 
$$b_2=2k+\frac{m\rho^2}{1+\rho^2}-\frac{|n|a\rho^2}{b(1+\rho^2)}-i\left(\frac{|n|(2+\rho^2-2\rho)}{b(1+\rho^2)}+\frac{\rho b}{4\pi(1+\rho^2)}+\frac{b}{4\pi}\right),$$
$$ b_3=2k\left(1+\rho^2\right)-m\rho^2+\frac{|n|a\rho^2}{b}+i\left(\left(2+\rho^2+2\rho\right)\frac{|n|}{b}+\frac{\rho b}{4\pi}-\left(1+\rho^2\right)\frac{b}{4\pi} \right).$$
We omit the values of $b_1$ and $b_4$ as they are irrelevant to the following calculations.

Applying Theorem \ref{1stode}, in order to have solutions of the above ODE systems for $n\ne 0$, we know
$$\frac{\pi\sqrt{1+\rho^2}}{4|n|}\left(\left(2k+\frac{2\rho |n|i}{b(1+\rho^2)}-\frac{bi}{4\pi} \right)^2-\left(\frac{m\rho^2}{1+\rho^2}- \frac{|n|a\rho^2}{b(1+\rho^2)}-\frac{|n|(2+\rho^2)i}{b(1+\rho^2)}-\frac{\rho bi}{4\pi(1+\rho^2)}\right)\right)$$ is a negative integer. 
We notice that as $k, m, n\in \mathbb Z$,  if $a\in \mathbb Q$ and $\rho$ is a rational multiple of $\pi$, then for any rational value of $d=\frac{b}{8\pi}$ there are no solutions to the above relation as $\pi$ is a transcendental number.

Hence, for these choices of $a, b, \rho$, all the solutions are linear combinations of those with $n=0$ of type 
\begin{align*}
    f = F_{k,l,m,0}e^{2\pi i (kt+lx+my)},\\
     g = G_{k,l,m,0}e^{2\pi i (kt+lx+my)}.
\end{align*}
 Hence the equations would be
\begin{align*}
    -mF_{k,l,m,0}+\left(k+il-\frac{bi}{4\pi}\right)G_{k,l,m,0}&=0,\\
    \left((1+\rho^2)(k-il)+\rho m\right)F_{k,l,m,0}+\left(m+\rho(k-il)\right)G_{k,l,m,0}&=0.
\end{align*}

When $m=0$, we still have the solution \eqref{sol 1}
$$f=C_0, \quad g=0.$$
This is the only case that $g$ is zero. Hence, in other situations, we can write $d=\frac{b}{8\pi}$ and cancel $f$ to get
\begin{align*}
    k^2+l^2+m^2-2dl+\rho^2 (k^2+l^2-2dl)+2\rho mk&=0,\\
    dk(1+\rho^2)+\rho md&=0.
\end{align*}

Then for any rational non-half-integer $d$, if we choose $\rho$ such that $\rho^2$ is irrational, there will be no solution other than $m=k=l=0$. But for $\rho=0$, there are many solutions by Theorem \ref{h01}. 

To summarise the cases of $n\ne 0$ and $n=0$, if we choose an almost complex structure $J_{a, b}$ on $\KT$ with $a\in \mathbb Q$ and $b=\frac{8\pi\cdot 5^{\frac{K-1}{2}}}{3}$ for any odd $K\in \mathbb Z^+$, then for almost Hermitian metric $h_0$, we have $h^{0,1}=K$; and for almost Hermitian metric $h_{\rho}$ with $\rho$ a rational multiple of $\pi$, we have $h^{0,1}=1$.

Hence, we have answered Question \ref{KS}.

\begin{thm}\label{KSKT}
There exist almost complex structures on $\KT$ such that $h^{0,1}$ varies with different choices of almost Hermitian metrics.
\end{thm}

In our construction, the almost Hermitian metrics $h_{\rho}$ are not almost K\"ahler when $\rho\ne 0$.
\begin{q}
Can we construct an almost complex structure $J$ on $\KT$, or more generally on a $4$-manifold, such that $h^{0,1}$ varies with different choices of almost K\"ahler metrics?\footnote{This question is answered affirmatively in \cite{HZ}.}
\end{q}

In addition to Question \ref{KS}, the second part of problem 20 in \cite{Hir} further asks for other definitions of Hodge numbers on almost complex manifolds, which generalise that of complex manifolds. There are some candidates, for example in \cite{CWb}, although it is not known whether this definition always results in finite Hodge numbers. A real version of Dolbeault cohomology for almost complex manifolds is defined in \cite{LZ} (see also \cite{DLZ}). 

\section{Computation of $\boldsymbol{h^{1,1}}$}\label{h11}
Finally, we will compute $h^{1,1}$ for $J_{a, b}$ with its standard orthonormal metric. 

\subsection{A general method for almost K\"ahler structures}
In turns out that $h^{1,1}$ is in fact a topological invariant when $J$ is almost K\"ahler on a $4$-manifold. In particular, on the Kodaira-Thurston manifold, when $J = J_{a,b}$ we always have $h^{1, 1}=3$.

\begin{prop}\label{akh11}
For any closed almost K\"ahler $4$-manifold $(M, J)$, the dimension of the space of $\bar\partial$-harmonic $(1, 1)$-forms $\mathcal H^{1,1}_{\bar\partial}$ is independent of the choice of almost K\"ahler metrics compatible with $J$. More precisely, $h^{1,1}=b^-+1$.
\end{prop}
Here $b^-$ denotes the dimension of anti-self-dual harmonic $2$-forms.  
\begin{proof}
Let $g$ be a $J$-compatible almost K\"ahler metric on $M$. The pair $(g, J)$ defines a $J$-invariant closed $2$-form $\omega$ by $\omega(u, v)=g(Ju, v)$.

Let $\Lambda_J^+$ be the bundle of real $2$-forms with $\alpha(JX, JY)=\alpha(X, Y)$ and $\Lambda_g^-$ the bundle of $g$-anti-self-dual forms, we have the bundle decomposition \cite{DLZ} $$\Lambda_J^+=\underline{\mathbb R}(\omega)\oplus \Lambda_g^-.$$ When we complexify it, we have $$\Lambda_J^{1,1}=\underline{\mathbb C}(\omega)\oplus (\Lambda_g^-\otimes \mathbb C).$$

Let $\mu$ denote the component of the exterior derivative which changes the bidegree of a $(p,q)$-form by $(+2,-1)$.
Then define the operator $\Delta_{\mu}:=\mu^*\mu+\mu\mu^*$ and the space $\mathcal{H}_{\mu}^{1,1}=\ker \Delta_{\mu}$. For bidegree reasons $\Delta_{\mu} = 0$ when acting on $(1,1)$-forms, hence we have $\mathcal H_{\bar\partial}^{1,1}=\mathcal H_{\bar\partial}^{1,1}\cap\mathcal H_{\mu}^{1,1}$. Furthermore, from Theorem 4.1 in \cite{CW} we know that $\mathcal H_{\bar\partial}^{1,1}\cap\mathcal H_{\mu}^{1,1} = \mathcal{H}_d^{1,1}$.
 Therefore, calculating $h^{1,1}$ is equivalent to calculating the dimension of the space of $d$-harmonic $(1,1)$-forms with respect to the almost K\"ahler metric $g$. We can write any harmonic $(1, 1)$-form as $\alpha+\beta$ where $\beta=f\cdot \omega$ and $\alpha\in \Omega_g^-\otimes \mathbb C$. We then have $$d\alpha+d\beta=0, \quad d*(\alpha+\beta)=0.$$ But we know $*\alpha=-\alpha$ and $*\beta=\beta$.  Hence, $d*(\alpha+\beta)=0$ implies $d\alpha=d\beta$. Combining with the first equation we have $d\alpha=d\beta=0$. 

As $\omega$ is non-degenerate, $d\beta=df\wedge\omega =0$ is equivalent to saying $f$ is a constant. Hence any harmonic $(1,1)$-form is a sum of a complex constant multiple of $\omega$ and a complexified  anti-self-dual harmonic form. Hence, $h^{1,1}=b^-+1$.
\end{proof}

On the Kodaira-Thurston manifold we have $b^- = b^+ = 2$, and therefore for any almost K\"ahler structure (\textit{e.g.} $J_{a,b}$ with the standard orthonormal metric), we have $h^{1,1} = 3$. From our calculation of $h^{0,2}$ and $h^{2,0}$ in Section \ref{hodge}, we observe that $h^{2,0}+h^{1,1}+h^{0,2}$ of $J_{a, b}$ is well defined, and could be $3$, $4$ (integrable case), or $5$, while $b_2=4$ for $\KT$. 

The almost K\"ahler condition is indispensable in Proposition \ref{akh11}. However, we would like to know whether this is still true in general when an almost Hermitian metric is fixed.

\begin{q}
Let $(M, J)$ be an almost complex $4$-manifold which admits an almost K\"ahler structure. Does there exist an almost Hermitian metric which is not almost K\"ahler, such that $h^{1,1}\ne b^-+1$? \footnote{This question is answered negatively in \cite{TT}, $h^{1,1}=b^-$ if the metric is strictly locally almost K\"ahler.}
\end{q}

The following identity was proven in \cite{CW} to hold for any smooth $(p,q)$-form $s$ on an almost K\"ahler manifold:
\begin{equation}\label{id}
  (\Delta_{\mu}+\Delta_{\bar{\partial}})s=(\Delta_{\bar{\mu}}+\Delta_{\partial})s,  
\end{equation}
 where $d=\bar\mu+\bar\partial+\partial+\mu$ decomposed with respect to bidegree.
 
This fact was never used in the calculations of $h^{p,q}$ in earlier sections, although this does provide us with a tool for use in future calculations.
We will consider $(0, 1)$-forms to illustrate how this identity might be used. 

\begin{prop}\label{prop 1}
On a compact almost K\"ahler manifold $M$, with real dimension 4, any $\bar\partial$-harmonic smooth $(0,1)$-form $s$ must satisfy
$$ \|\partial s\|^{2}= \|\mu s \|^{2}.$$
\begin{proof}
When acting on a $(0,1)$-form $\bar\mu, \partial^*$ and $\mu^*$ are all zero for bidegree reasons. If a $(0,1)$-form is $\bar\partial$-harmonic, then $\bar\partial$ and $\bar\partial^*$ are also zero. Using all of this, we can simplify \eqref{id} show that a $\bar\partial$-harmonic $(0,1)$-form satisfies
$$\partial^* \partial s = \mu^* \mu s. $$

Let $g$ denote the almost K\"ahler metric on $M$. There is a natural way to extend $g$ to differential forms. We can then define the inner product
$$(\cdot, \cdot) = \int_{\KT}g(\cdot, \cdot) dV. $$
It is with respect to this inner product that the adjoint operators $\partial^*$ and $\mu^*$ are defined.

Taking the inner product of the identity $\partial^* \partial s = \mu^* \mu s $ with $s$ and making use of the property $$(\mu^{*}\mu s,s)=(\mu s,\mu s) \quad \text{ and } \quad (\partial^{*}\partial s,s)=(\partial s,\partial s) $$
we conclude that
$$ \|\partial s\|^{2}= \|\mu s \|^{2} .$$
\end{proof}
\end{prop}

This proposition can be applied to the example of the Kodaira-Thurston manifold, equipped with the almost complex structure $J_{a,b}$ and the standard orthonormal metric to obtain an upper bound on the size of $n$ for which $f,g \in \mathcal{H}_{k,m,n}$ can still give solutions to $\eqref{V1}$ and $\eqref{V2}$.

\begin{prop}
On the almost complex manifold $(\KT, J_{a,b})$, with the standard orthonormal metric, there are no $\bar\partial$-harmonic $(0,1)$-forms given by $f\bar\phi_1 + g \bar\phi_2$ with $f,g \in \mathcal{H}_{k,m,n}$ for any $k,m\in \mathbb{Z}, 0\leq m < \abs{n}$, unless $\abs{n}\leq \frac{\sqrt 2 b^2}{8\pi}$.
\begin{proof}
Recall that we can write a general $\bar\partial$-harmonic $(0,1)$-form as $s = f \bar\phi_1 + g\bar\phi_2$, with $f,g \in C^{\infty}(\KT)$. Writing out $\mu s$ and $\partial s$ explicitly Prop. \ref{prop 1} tells us that
\begin{align*}
    \int_{\KT}\left(\abs{V_1(f)}^2+\abs{V_2(g)}^2+\abs{V_2(f)-\frac{b}{4}g}^2+\abs{V_1(g)-\frac{b}{4}g}^2\right)dV
    =\int_{\KT}\left(\abs{\frac{b}{4}g}^2\right)dV.
\end{align*}
Here the integral is taken over $[0,1]^4$ as it is the fundamental domain on $\mathbb{R}^4$ with respect to the identification \eqref{quot}.  Since $|V_1(f)|=|V_2(g)|$ by \eqref{V2}, we have an inequality 
\begin{equation}
    \int_{\KT}2|V_2(g)|^2 dV\leq\int_{\KT}\left(\abs{\frac{b}{4}g}^2\right)dV.\label{ineq}
\end{equation}
If we now assume $f,g \in \mathcal{H}_{k,m,n}$ then this inequality becomes
$$\int_{\KT}2\pi^{2} \abs{\left(m+nx-\frac {a-i}{b}n\right)g}^{2}dV\leq \int_{\KT}\frac{b^{2}}{16}\sum_{m\in\mathbb{Z}}|g|^{2}dV. $$
In particular, if we discard the real part of $m+nx-\frac {a-i}{b}n$ we are left with
$$\frac{2\pi^{2}n^{2}}{b^{2}}\int_{\KT}\sum_{m\in\mathbb{Z}}|g|^{2}dV\leq \frac{b^{2}}{16}\int_{\KT}\sum_{m,n\in\mathbb{Z}}|g|^{2}dV ,$$
giving us the bound $|n|\leq\frac{\sqrt{2}b^2}{8\pi}$ on the largest value of $n$ we need to check for solutions. 
\end{proof}
\end{prop}

\subsection{An alternative method}
We should also be able to use our PDE-to-ODE method for $h^{0,1}$ to calculate $h^{1,1}$ for $J_{a, b}$ and the standard orthonormal metrics. 

We write a general $(1,1)$-form $s$ and its Hodge star $*s$ as 
\begin{align*}
    s&=f^{(1,1)}\phi_{1}\wedge\bar{\phi}_{1}+f^{(1,2)}\phi_{1}\wedge\bar{\phi}_{2}+f^{(2,1)}\phi_{2}\wedge\bar{\phi}_{1}+f^{(2,2)}\phi_{2}\wedge\bar{\phi}_{2},\\
    *s &=-f^{(2,2)}\phi_{1}\wedge\bar{\phi}_{1}+f^{(1,2)}\phi_{1}\wedge\bar{\phi}_{2}+f^{(2,1)}\phi_{2}\wedge\bar{\phi}_{1}+f^{(1,1)}\phi_{2}\wedge\bar{\phi}_{2}.
\end{align*}
We then calculate $\bar{\partial} s$ and $\partial *s$, requiring both to be zero. Looking at the   $\phi_{1}\wedge\bar{\phi}_{1}\wedge\bar{\phi}_{2}$ and $\phi_{2}\wedge\bar{\phi}_{1}\wedge\bar{\phi}_{2}$ components of $\bar{\partial}s$ separately yields 
$$\bar{V}_{2}(f^{(1,1)})-\bar{V}_{1}(f^{(1,2)})-\frac b4 \left(f^{(1,2)}+f^{(2,1)}\right)=0,$$
$$\bar{V}_{2}(f^{(2,1)})-\bar{V}_{1}(f^{(2,2)})=0.$$
Similarly, looking at the $\phi_{1}\wedge\phi_{2}\wedge\bar{\phi}_{1}$ and $\phi_{1}\wedge\phi_{2}\wedge\bar{\phi}_{2}$ components of $\partial * s$ yields 
$$V_{2}(f^{(2,2)})+V_{1}(f^{(2,1)})+\frac b4\left(f^{(1,2)}+f^{(2,1)}\right)=0,$$
$$V_{2}(f^{(1,2)})+V_{1}(f^{(1,1)})=0. $$
As in our calculation of $h^{0,1}$ we use \eqref{decomp} to look for solutions given by $f^{(1,1)},f^{(1,2)},f^{(2,1)},f^{(2,2)}\in\mathcal{H}_{k,l,m,0}$ and $\mathcal{H}_{k,m,n}$.

 When looking for solutions in $\mathcal{H}_{k,l,m,0}$, we write
 $$f^{(i,j)} = G^{(i,j)}e^{2\pi i (kt+lx+my)}. $$
 Substituting this into the PDE system yields the following
\begin{align*}
     &mG^{(1,1)}-(k+il-\frac{bi}{4\pi})G^{(1,2)}+\frac{bi}{4\pi}G^{(2,1)}=0, \\
    &mG^{(2,1)}-(k+il)G^{(2,2)}=0,\\ 
    &mG^{(2,2)}+(k-il-\frac{bi}{4\pi})G^{(2,1)}-\frac{bi}{4\pi}G^{(1,2)}=0, \\
    &mG^{(1,2)}+(k-il)G^{(1,1)}=0.   
\end{align*}
It is relatively simple to see that when $k=l=m=0$ the above equations tell us that $G^{(1,2)}=-G^{(2,1)}$ and give no restrictions on $G^{(1,1)}$ and $G^{(2,2)}$. Therefore we have a family of solutions given by
$$f^{(1,1)}=C_{0},\quad f^{(1,2)}=C_{1},\quad f^{(2,1)}=-C_{1},\quad f^{(2,2)}=C_{2} $$
for any three constants $C_{0}, C_{1}, C_{2}\in\mathbb{C}$.

Since $J_{a,b}$ is always almost K\"ahler, Proposition \ref{akh11} implies $h^{1,1}=3$, so there are no solutions to be found. We should be able to directly derive this by solving the ODE systems obtained when looking for solutions in $\mathcal{H}_{k,m,n}$, as in the calculation of $h^{0,1}$. We expect to have a general result similar to Theorem \ref{1stode} which works for $N\times N$ systems, in particular when $N=4$. However, as contrasted with Proposition \ref{sol>0}, we will not have any Schwartzian solutions to the ODE systems derived from our current setting.

\end{document}